\documentclass[a4paper,10pt]{article}

\usepackage{amsfonts}
\usepackage{amsmath}
\usepackage{graphicx}
\usepackage{color}
\usepackage{float}
\usepackage{hyperref}

\newenvironment{proof}[1][Proof]{\begin{trivlist}
\item[\hskip \labelsep {\bfseries #1}]}{\end{trivlist}}
\newenvironment{definition}[1][Definition]{\begin{trivlist}
\item[\hskip \labelsep {\bfseries #1}]}{\end{trivlist}}

\newtheorem{theorem}{Theorem}[section]
\newtheorem{lemma}[theorem]{Lemma}

\newcommand{\qed}{\nobreak \ifvmode \relax \else
      \ifdim\lastskip<1.5em \hskip-\lastskip
      \hskip1.5em plus0em minus0.5em \fi \nobreak
      \vrule height0.75em width0.5em depth0.25em\fi}

\newcommand{\si}{\sigma}

\title{Fibered and primitive/Seifert twisted torus knots}
\author{Brandy Guntel Doleshal}

\begin{document}

\maketitle{}

\begin{abstract}
The twisted torus knots lie on the standard genus 2 Heegaard surface for $S^3$, as do the primitive/primitive and primitive/Seifert knots. It is known that primitive/primitive knots are fibered, and that not all primitive/Seifert knots are fibered. Since there is a wealth of primitive/Seifert knots that are twisted torus knots, we consider the twisted torus knots to partially answer the question of which primitive/Seifert knots are fibered. A braid computation shows that a particular family of twisted torus knots is fibered, and that computation is then used to generalize the results of a previous paper by the author.
\end{abstract}

\section{Introduction}

In \cite{berge}, John Berge introduced the primitive/primitive knots and noted that they have lens space surgeries. Later, John Dean \cite{dean} described a generalization, called primitive/Seifert knots, and observed that the primitive/Seifert knots have Seifert fibered surgeries. (The definitions of primitive/primitive and primitive/Seifert can be found at the beginning of Section \ref{bigonesection}.)

Since it was not clear that the knots in \cite{berge} were all of the primitive/primtive knots, Berge's list of knots became known as Berge knots. In \cite{osvathszabo}, Ozsv{\'a}th and Szab{\'o} show that the Berge knots are fibered (i.e. the complement of each knot, in $S^3$, is a fiber bundle over the circle). Later, in the classification of primitive/primitive and primitive/Seifert knots \cite{bgk}, it was shown that the Berge knots are in fact all the primitive/primitive knots. Hence primitive/primitive knots are fibered. Alternatively, work of Ni \cite{yini} shows that knots with lens space surgeries (e.g. primitive/primitive knots) are fibered.  

On the other hand, knots admitting Seifert fibered surgeries are not all fibered, but there is no classification of those that are fibered and those that are not. In \cite{dean}, Dean describes a family of twisted torus knots that are primitive/Seifert: $K(p,q, p-kq, -1)$ where $2 \le q \le \frac{p}{2}$ and $k$ is an integer with $2 \le k \le \frac{p-2}{q}$. Since Dean's knots are easily described as braids, they are a good first family of knots to explore when one would like to describe fibered primitive/Seifert knots. In fact, when a knot in this family satisfies the condition $p-kq<q$, the knot is fibered, and more generally, the twisted torus knot $K(p,q,r,-n)$ with $n>0$ and $r<q$ is fibered when $nq<p$, as shown in Theorem \ref{fiberedttks}. The proof of Theorem \ref{fiberedttks} uses braid group calculations to show that the knot in question has a positive braid representative, and is, therefore, fibered. As of this writing, the only known knots in Theorem \ref{fiberedttks} with a primitive/Seifert representative are the ones in Theorem \ref{thebigone}: $K(p,q, p-kq, -1)$ where $2 \le q \le \frac{p}{2}$ and $k$ is an integer with $2 \le k \le \frac{p-2}{q}$ and $p-kq <q$. (For ease of calculation, the parameters are expressed slightly differently in Theorem \ref{thebigone} than in Dean's work.)

The word for the positive braid representative found in the proof of Theorem \ref{fiberedttks} is a tidier way of describing the knots in that theorem. Since Dean's primitive/Seifert twisted torus knots that have $p-kq <q$ and $n=1$ are a special case of the knots in Theorem \ref{fiberedttks}, we not only know that they are fibered but we can use the positive braid word to expand on previous work by this author \cite{mypaper}. In Section \ref{bigonesection}, we use the word from the proof of Theorem \ref{fiberedttks} to prove Theorem \ref{thebigone}: there is an infinite family of knots that have either two primitive/Seifert representatives with the same surface slope or a primitive/primitive and a primitive/Seifert representative sharing the same surface slope. Theorem \ref{thebigone} provides a mutual generalization of two theorems in \cite{mypaper}. 

Using different methods than those employed here, Eudave-Mu\~noz, Miyazaki and Motegi find a family of knots with Seifert fibered surgeries and distinct primitive/Seifert positions with the same slope \cite{emmm}. Their knots are twisted torus knots of the form $K(p,q,p+q,n)$, where $|n|\ge2$, so they are distinct from the family of knots with the same properties in Theorem \ref{thebigone}.

The author would like to thank Cameron Gordon for valuable conversations and suggestions and the reviewer for helpful input. This work is partially supported by NSF RTG Grant DMS-0636643.

\section{Fibered knots and braids}\label{fkb}

Some knots have complements with added structure, giving us more to work with while studying them. One example of this is the class of fibered knots.

\begin{definition}
A knot $K$ in $S^3$ is fibered if $S^3 - K$ is homeomorphic to $(F \times I )/f$, where $F$ is the interior of a Seifert surface for $K$ and the map $f: F\times \{0\} \rightarrow F\times \{1\}$ is a homeomorphism. 
\end{definition}

Here, we give a family of twisted torus knots, each of which is fibered. The definition of the twisted torus knot can be traced through Dean's work \cite{dean} (with the requirement that $m=1$ in his description of $K(p,q,r,m,n)$).

\begin{definition}
The \textit{twisted torus knot} $K(p,q,r,n)$ is obtained from the torus knot $T(p,q)$ by twisting $r$ strands of $T(p,q)$ $n$ full twists.
\end{definition}

 The twisted torus knot can be viewed as a curve lying in the standard genus 2 Heegaard surface for $S^3$ in the following way. Let $D$ be a disk in the torus so that $T(p,q)$ intersects $D$ in $r$ disjoint arcs, for $0\le r\le p+q$, which are parallel in $D$ and coherently oriented. Let $r$ parallel coherently-oriented copies of the torus knot $T(1,n)$, denoted $rT(1,n)$, lie in a torus so that each copy of $rT(1,n)$ intersects a disk $D'$ in the torus in one arc. (These arcs will be parallel in $D'$ and coherently oriented.) That is, $rT(1,n)$ intersects $D'$ in $r$ disjoint arcs, one for each component of $rT(1,n)$. We excise the disks $D$ and $D'$ from their respective tori and glue the punctured tori together along their new boundaries so that the orientations of the torus links $T(p,q)$ and $rT(1,n)$ align correctly. Figure \ref{ttk} shows the twisted torus knot $K(3, 5 ,2,-1)$ before the excision of $D$ and $D'$.

\begin{figure}[ht]\label{ttk}
\begin{center}
\includegraphics[scale=.3]{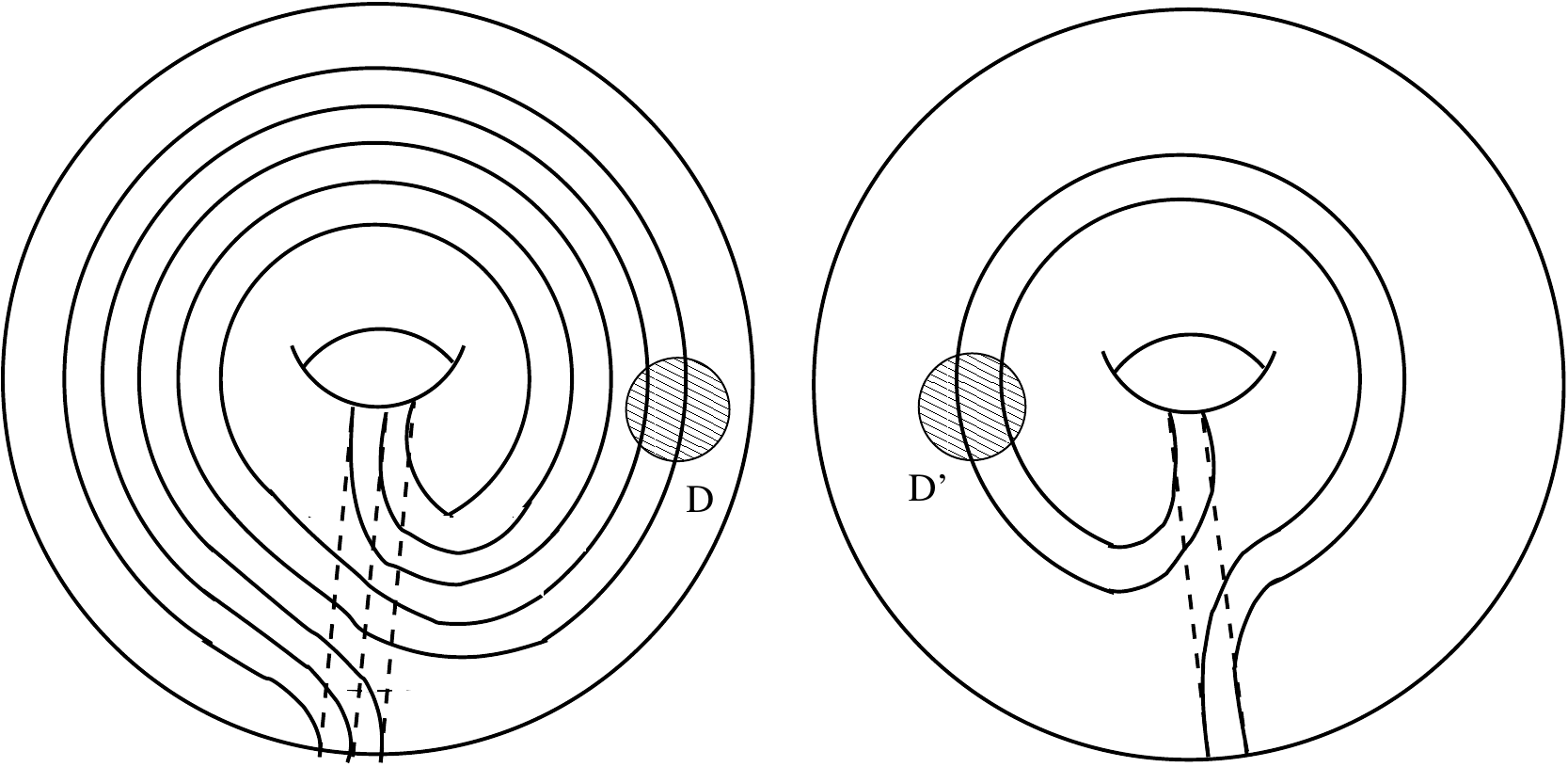}
\caption{$K(3, 5 ,2,-1)$}
\end{center}
\end{figure}

Since $K(p,q,r,n)$ typically has many crossings, considering the Seifert surface may be quite difficult. Fortunately, the twisted torus knots are easily seen as closures of braids. The torus knot $T(p,q)$ is represented by the braid $(\si_{q-1} \si_{q-2} \cdots \si_1)^p$. Since $K(p,q,r,n)$, with $r<p$ and  $r<q$, is obtained from $T(p,q)$ by twisting $r$ strands $n$ full twists, $K(p,q,r,n)$ is represented by the braid $(\si_{q-1} \si_{q-2} \cdots \si_1)^p (\si_{r-1} \si_{r-2} \cdots \si_1)^{nr}$ in $B_q$. To simplify symbols, we use a modified version of the notation of Garside \cite{gars}, first introduced in \cite{mypaper}.

We use the symbol $\Pi_{s}^l$ to denote $\si_l \si_{l+1} \cdots \si_s$, and we use the symbol $\Delta_s^l$ to denote $\Pi_s^l \Pi_{s-1}^l \cdots \Pi_l^l$. When $l=1$, we drop the superscript, and to avoid confusion with exponents, any exponent will occur outside of parentheses. As in \cite{gars}, if $w$ is a braid word representing a braid in the braid group $B_n$, $\mbox{rev}w$ denotes the braid word obtained by writing the generators appearing in $w$ in the reverse order. For example, in $B_4$, the reverse of $w = \si_1 \si_2 \si_3$ is $\mbox{rev}w = \si_3 \si_2 \si_1$.    

It is a result of Stallings \cite{stallings} that if a braid $\beta$ is \textit{homogeneous}, then the closure of $\beta$ is a fibered link. Here, homogeneous means that every generator of the braid group appears in the word for $\beta$ only with exponents of the same sign. In particular, if a braid can be represented by a word with only positive (or only negative) exponents on the generators, the braid is homogeneous. In Section \ref{fiberedttksection}, we will use Stallings' result to show that a certain class of twisted torus knots are fibered.

\section{Fibered Twisted Torus Knots}\label{fiberedttksection}

In this section, we demonstrate that a specific class of twisted torus knots are fibered. Since the braid representing $K(p,q,r,n)$ is homogeneous for $n>0$, all twisted torus knots $K(p,q,r,n)$ are fibered when $n>0$, so we assume $n<0$. In Theorem \ref{fiberedttks}, we assume $n>0$ and use the symbol $-n$ to be a negative number, so as to simplify calculations.

\begin{theorem}\label{fiberedttks}
The twisted torus knots $K(p,q,r,-n)$, with $n>0$ and $r<q$, are fibered when $nq<p$.
\end{theorem}


The primitive/Seifert twisted torus knots in Theorem 4.1 of \cite{dean} are a special case of the knots in Theorem \ref{fiberedttks} (with $n=1$), so we know that they are also fibered. The proof of Theorem \ref{fiberedttks} leads to an explicit positive braid word for $K(p,q,r,-1)$, when $r <p$ and $r<q$. In Section \ref{bigonesection}, we use this explicit braid representation to provide a mutual generalization of the two main theorems (3.1 and 4.1) of a previous paper written by the author \cite{mypaper}. 

No obvious strengthening of Theorem \ref{fiberedttks} exists, as there is a simple example of a twisted torus knot which is not fibered. Consider the knot $K(4,3,2,-2)$, represented by the closure of the braid $\beta =(\si_2 \si_1)^4 (\si_1)^{-4}$, seen in Figure \ref{twistknot1}. The braid $\beta$ can easily be seen, geometrically, to be equivalent to the braid $(\si_2)^2 (\si_1)^2 \si_2 \si_1^{-1}$. The closure of this braid, seen in Figure \ref{twistknot2}, is isotopic to the 5-crossing twist knot. Since twist knots with 5 or more crossings are not fibered \cite{knots}, $K(4,3,2,-2)$ is not a fibered knot.  

\begin{figure}[ht]
\begin{center}
\includegraphics[scale=.7]{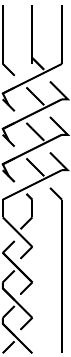} 
\caption{$K(4,3,2,-2)$}\label{twistknot1}
\end{center}
\end{figure}

\begin{figure}[h]
\begin{center}
\includegraphics[scale=.7]{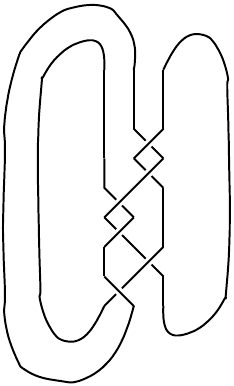} 
\caption{$K(4,3,2,-2)$}\label{twistknot2}
\end{center}
\end{figure}

To prove Theorem \ref{fiberedttks}, we first prove two lemmas about words in the braid group, and state a lemma, introduced in \cite{mypaper}.

\begin{lemma}\label{lemma!}
For $l < s$, $\si_l^{-1} \Pi_{s}^{l+1} \Pi_{s}^{l} = \Pi_{s}^{l+1} \Pi_{s-1}^{l}$
\end{lemma}

\begin{proof}
First, we note that the relation of the braid group $\si_i \si_{i+1} \si_i = \si_{i+1} \si_i \si_{i+1}$ is equivalent to $\si_{i}^{-1} \si_{i+1} \si_{i} = \si_{i+1} \si_i \si_{i+1}^{-1}$. The proof uses this relation repeatedly, as well as the commutativity properties of non-adjacent generators of the braid group.
\begin{eqnarray*}
\si_l^{-1} \Pi_s^{l+1} \Pi_{s}^{l} & = & \si_l^{-1} \si_{l+1} \si_{l+2} \cdots \si_s \si_{l} \si_{l+1} \cdots \si_s\\
                      & = &  \si_l^{-1} \si_{l+1} \si_l \si_{l+2} \cdots \si_s \si_{l+1} \si_{l+2} \cdots \si_s \\
                      & = &  \si_{l+1} \si_l \si_{l+1}^{-1} \si_{l+2} \cdots \si_s \si_{l+1} \si_{l+2} \cdots \si_s \\
                      & = &  \si_{l+1} \si_l \si_{l+1}^{-1} \si_{l+2} \si_{l+1} \si_{l+3} \cdots \si_s \si_{l+2} \si_{l+3} \cdots \si_s \\
                       & = &  \si_{l+1} \si_l \si_{l+2} \si_{l+1} \si_{l+2}^{-1} \si_{l+3} \cdots \si_s \si_{l+2} \si_{l+3} \cdots \si_s\\     
                      & = & \si_{l+1} \si_l \si_{l+2} \si_{l+1} \si_{l+3} \si_{l+2} \cdots \si_{s} \si_{s-1} \si_s^{-1} \si_s \\
                      & = & \si_{l+1} \si_l \si_{l+2} \si_{l+1} \si_{l+3} \si_{l+2} \cdots \si_{s} \si_{s-1}
\end{eqnarray*}
By the commutativity properties of non-adjacent generators of the braid group, we can rewrite this element as $\si_{l+1} \cdots \si_{s} \si_{l} \cdots \si_{s-1}$, which is $\Pi_s^{l+1} \Pi_{s-1}^l$.
\end{proof}

\begin{lemma}\label{easylemma}
For $l<s$, $\si_l^{-1} \Pi_s^{l+1} \Pi_{s-1}^l = \Pi_s^{l+1} \Pi_{s-1}^l \si_s^{-1}$.
\end{lemma}

To prove Lemma \ref{easylemma} rewrite $\si_l^{-1} \Pi_s^{l+1} \Pi_{s-1}^l$ as $\si_l^{-1} \Pi_s^{l+1} \Pi_{s}^l \si_s^{-1}$. Lemma \ref{lemma!} completes the proof.

\begin{lemma}\label{L1}
For $l <  t \leq s$,  $\si_t  \Pi_s^l = \Pi_s^l \si_{t-1}$.
\end{lemma}

The proof of Lemma \ref{L1}, which uses relations of the braid group, can be found in \cite{mypaper}.

\begin{proof}[Proof of Theorem \ref{fiberedttks}.]
We first prove the theorem in the special case $n=1$.
Because $K(p,q,r,-1)$ is a twisted torus knot with $r<p$ and $r<q$, we can think of it as the closure of the braid on $q$ strands, shown in Figure \ref{genttk}. In this picture, each box represents the indicated number of single (non-full) twists on the strands that enter that box. Here, a box with 1 in it would be represented by a braid in which the right-most strand crosses over each of the other strands once, in order from right to left. 

Using the modified Garside notation, explained above, we can write this braid as $(\mbox{rev}\Pi_{q-1})^p (\mbox{rev}\Pi_{r-1})^{-r}$. The goal in this proof is to show that the braid representing $K$ has an expression as a homogeneous word (in this case, a positive word). For any $\gamma \in B_n$, it is easy to see that $\mbox{rev}\gamma$ is homogeneous if and only if $\gamma$ is homogeneous. Hence we work instead with $\beta = $rev$((\mbox{rev}\Pi_{q-1})^p (\mbox{rev}\Pi_{r-1})^{-r}) = (\Pi_{r-1})^{-r} (\Pi_{q-1})^p$. The proof is now an application of Lemma \ref{lemma!} and the properties of the braid group.

\begin{figure}[ht]
\begin{center}
\includegraphics[scale=.65]{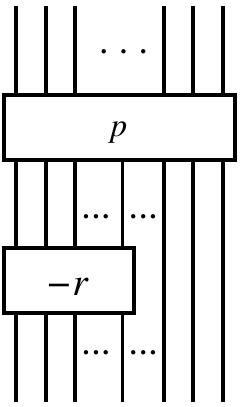} 
\caption{$K(p,q,r,-1)$}\label{genttk}
\end{center}
\end{figure}

For the first step, observe that $(\Pi_{r-1})^{-1} \Pi_{q-1} = \Pi_{q-1}^r$. Using this relation changes $\beta$ to
$(\Pi_{r-1})^{-r+1}  \Pi_{q-1}^{r} (\Pi_{q-1})^{p-1}$. Since the goal is to show $\beta$ is positive, we examine how the negative powers of $\Pi_{r-1}$ interact with the rest of the word; begin by stripping away powers of $(\Pi_{r-1})^{-1}$, one at a time. We have $\beta = (\Pi_{r-1})^{-r+2} (\Pi_{r-1})^{-1}  \Pi_{q-1}^{r} (\Pi_{q-1})^{p-1}$. Using the commutativity properties of non-adjacent generators of the braid group, $\beta$ can be written as $(\Pi_{r-1})^{-r+2} \si_{r-1}^{-1} \Pi_{q-1}^{r} (\Pi_{r-2})^{-1} (\Pi_{q-1})^{p-1}.$ Because $(\Pi_{r-2})^{-1} \Pi_{q-1} = \Pi_{q-1}^{r-1}$, $\beta$ becomes $(\Pi_{r-1})^{-r+2} \si_{r-1}^{-1} \Pi_{q-1}^{r}\Pi_{q-1}^{r-1} (\Pi_{q-1})^{p-2}$. Now Lemma \ref{lemma!} tells us that this is $(\Pi_{r-1})^{-r+2} \Pi_{q-1}^{r} \Pi_{q-2}^{r-1} (\Pi_{q-1})^{p-2}$.

Now strip away another copy of $(\Pi_{r-1})^{-1}$ and another copy of $\Pi_{q-1}$ to obtain the word $(\Pi_{r-1})^{-r+3} (\Pi_{r-1})^{-1} \Pi_{q-1}^{r} \Pi_{q-2}^{r-1} \Pi_{q-1} (\Pi_{q-1})^{p-3}$. We repeat the previous process. This time, we can push all but one of the elements of $(\Pi_{r-1})^{-1}$ past $\Pi_{q-1}^{r}$ and all but one of the elements of $(\Pi_{r-2})^{-1}$ past $\Pi_{q-2}^{r-1}$. Then the right hand side of the equation becomes:          
\begin{eqnarray*}             
    (\Pi_{r-1})^{-r} (\Pi_{q-1})^p & = & (\Pi_{r-1})^{-r+3} \si_{r-1}^{-1} \Pi_{q-1}^{r} (\Pi_{r-2})^{-1} \Pi_{q-2}^{r-1} \Pi_{q-1} (\Pi_{q-1})^{p-3}\\
                      & = & (\Pi_{r-1})^{-r+3} \si_{r-1}^{-1} \Pi_{q-1}^{r} \si_{r-2}^{-1} \Pi_{q-2}^{r-1} (\Pi_{r-3})^{-1} \Pi_{q-1}(\Pi_{q-1})^{p-3}
\end{eqnarray*}    

From here, $\beta$ is written as $(\Pi_{r-1})^{-r+3} \si_{r-1}^{-1} \Pi_{q-1}^{r} \si_{r-2}^{-1} \Pi_{q-2}^{r-1} \Pi_{q-1}^{r-2}(\Pi_{q-1})^{p-3}$. For each $\si_i^{-1}$ appearing alone in the word, we will use Lemma \ref{lemma!}. Note, however that in order to use the lemma, we will need to rewrite $ \si_{r-2}^{-1} \Pi_{q-2}^{r-1} \Pi_{q-1}^{r-2}$ as $ \si_{r-2}^{-1} \Pi_{q-2}^{r-1} \Pi_{q-2}^{r-2} \si_{q-1}$.                 
Then $\beta = (\Pi_{r-1})^{-r+3} \si_{r-1}^{-1} \Pi_{q-1}^{r} \Pi_{q-2}^{r-1} \Pi_{q-3}^{r-2} \si_{q-1}(\Pi_{q-1})^{p-3}$. We push $\si_{q-1}$ to the left, past all elements with which it commutes, to obtain $(\Pi_{r-1})^{-r+3} \si_{r-1}^{-1} \Pi_{q-1}^{r} \Pi_{q-2}^{r-1} \si_{q-1}\Pi_{q-3}^{r-2} (\Pi_{q-1})^{p-3} $ which is easily seen to be the same as $(\Pi_{r-1})^{-r+3} \si_{r-1}^{-1} \Pi_{q-1}^{r} \Pi_{q-1}^{r-1}\Pi_{q-3}^{r-2} (\Pi_{q-1})^{p-3}$. Again, Lemma \ref{lemma!} tells us that we can exchange this word for $(\Pi_{r-1})^{-r+3} \Pi_{q-1}^{r} \Pi_{q-2}^{r-1}\Pi_{q-3}^{r-2} (\Pi_{q-1})^{p-3}$.

This process will repeat with each $(\Pi_{r-1})^{-1}$ we move toward the grouping of $\Pi_{q-1}$, so that after we have moved $l$ copies of $(\Pi_{r-1})^{-1}$, $\beta$ can be written as $ (\Pi_{r-1})^{-r+l}  \Pi_{q-1}^{r} \Pi_{q-2}^{r-1} \cdots \Pi_{q-l}^{r-l+1} (\Pi_{q-1})^{p-l}$. When $l = r-1$, we get  $(\Pi_{r-1})^{-1}  \Pi_{q-1}^{r} \Pi_{q-2}^{r-1} \cdots \Pi_{q-r+1}^{2} (\Pi_{q-1})^{p-r+1}$. After sliding each $\si_i$ in the final copy of $(\Pi_{r-1})^{-1}$ to the right as far as possible, $\beta$ is rewritten as the word $ \si_{r-1}^{-1}  \Pi_{q-1}^{r} \si_{r-2}^{-1} \Pi_{q-2}^{r-1} \cdots \si_1^{-1} \Pi_{q-r+1}^{2} (\Pi_{q-1})^{p-r}$. At this point, we want to use Lemma \ref{lemma!}, but the word is not quite in the right form to do so, so we pull off one copy of $\Pi_{q-1}$ of the $p-r$ copies available and use $\Pi_{q-1} = \Pi_{q-r+1} \Pi_{q-1}^{q-r+2}$ to obtain $\beta=  \si_{r-1}^{-1}  \Pi_{q-1}^{r} \si_{r-2}^{-1} \Pi_{q-2}^{r-1} \cdots \si_1^{-1} \Pi_{q-r+1}^{2} \Pi_{q-r+1} \Pi_{q-1}^{q-r+2} (\Pi_{q-1})^{p-r}$.

Now repeating the same process as above for the final time, we use Lemma \ref{lemma!} and the commutativity properties of non-adjacent generators of $B_q$ to get the following sequence of equivalent words for $\beta$. \begin{eqnarray*}  
\beta & = & \si_{r-1}^{-1}  \Pi_{q-1}^{r} \si_{r-2}^{-1} \Pi_{q-2}^{r-1} \cdots \si_2^{-1} \Pi_{q-r+2}^3  \Pi_{q-r+1}^{2} \Pi_{q-r} \Pi_{q-1}^{q-r+2} (\Pi_{q-1})^{p-r} \\
& = & \si_{r-1}^{-1}  \Pi_{q-1}^{r} \si_{r-2}^{-1} \Pi_{q-2}^{r-1} \cdots \si_2^{-1} \Pi_{q-r+2}^3  \Pi_{q-r+1}^{2} \si_{q-r+2} \Pi_{q-r} \Pi_{q-1}^{q-r+3} (\Pi_{q-1})^{p-r}\\
& = & \si_{r-1}^{-1}  \Pi_{q-1}^{r} \si_{r-2}^{-1} \Pi_{q-2}^{r-1} \cdots \si_2^{-1} \Pi_{q-r+2}^3  \Pi_{q-r+2}^{2} \Pi_{q-r} \Pi_{q-1}^{q-r+3} (\Pi_{q-1})^{p-r} \\
& = &  \si_{r-1}^{-1}  \Pi_{q-1}^{r} \si_{r-2}^{-1} \Pi_{q-2}^{r-1} \cdots \Pi_{q-r+2}^3  \Pi_{q-r+1}^{2} \Pi_{q-r} \Pi_{q-1}^{q-r+3} (\Pi_{q-1})^{p-r}  \\
& = &  \Pi_{q-1}^{r} \Pi_{q-2}^{r-1} \cdots \Pi_{q-r+1}^2  \Pi_{q-r} (\Pi_{q-1})^{p-r}
\end{eqnarray*} This braid word is positive, so its closure, $K(p,q,r,-1)$, is a fibered knot.

We will build the proof in the case $n>1$ on the proof in the special case $n=1$.
As noted above, the twisted torus knot $K(p,q,r,-n)$ is obtained from the $(p,q)$-torus knot by twisting $r$ consecutive strands $-n$ full twists. In the braid group, $K(p,q,r,-n)$ is represented by $(\mbox{rev}\Pi_{q-1})^p (\mbox{rev} \Pi_{r-1})^{-nr}$. Again, because taking the reverse word will not change whether the knot is homogeneous, we instead consider the braid $\beta =( \Pi_{r-1})^{-nr} (\Pi_{q-1})^p$. As a starting point, we use the braid word from the special case $n=1$ above, $(\Pi_{r-1})^{-r} (\Pi_{q-1})^p = \Pi_{q-1}^{r} \Pi_{q-2}^{r-1} \cdots \Pi_{q-r+1}^2  \Pi_{q-r} (\Pi_{q-1})^{p-r}$, to rewrite $\beta$ as the word $ ( \Pi_{r-1})^{-nr+r} \Pi_{q-1}^{r} \Pi_{q-2}^{r-1} \cdots \Pi_{q-r+1}^2  \Pi_{q-r}(\Pi_{q-1})^{p-r}$. 

Consider the subword $\gamma = ( \Pi_{r-1})^{-1} \Pi_{q-1}^{r} \Pi_{q-2}^{r-1} \cdots \Pi_{q-r+1}^2  \Pi_{q-r}$ of $\beta$. Write $\gamma$ so that the generators that make up the $(\Pi_{r-1})^{-1}$ are pushed to the right as far as possible, i.e. as $\gamma = \si_{r-1}^{-1} \Pi_{q-1}^{r} \si_{r-2}^{-1} \Pi_{q-2}^{r-1} \cdots \si_1^{-1} \Pi_{q-r+1}^2  \Pi_{q-r}$. By Lemma \ref{easylemma}, $\gamma = \Pi_{q-1}^{r}  \Pi_{q-2}^{r-1} \si_{q-1}^{-1} \cdots  \Pi_{q-r+1}^2  \si_{q-r+2}^{-1} \Pi_{q-r} \si_{q-r+1}^{-1}$. Now we push the generators to the right so that $\gamma= \Pi_{q-1}^{r}  \Pi_{q-2}^{r-1}  \cdots  \Pi_{q-r+1}^2 \Pi_{q-r} \si_{q-1}^{-1} \si_{q-r+2}^{-1} \cdots \si_{q-r+1}^{-1}$, which can be rewritten as  $\Pi_{q-1}^{r}  \Pi_{q-2}^{r-1}  \cdots  \Pi_{q-r+1}^2 \Pi_{q-r} (\Pi_{q-1}^{q-r+1})^{-1}$.

We can push each $(\Pi_{r-1})^{-1}$ past the $\Pi_{q-1}^{r} \Pi_{q-2}^{r-1} \cdots \Pi_{q-r+1}^2  \Pi_{q-r}$ section of $\beta$, so that $\beta =  \Pi_{q-1}^{r} \Pi_{q-2}^{r-1} \cdots \Pi_{q-r+1}^2  \Pi_{q-r}(\Pi_{q-1}^{q-r+1})^{-nr+r}(\Pi_{q-1})^{p-r}$. From Lemma \ref{L1}, $\si_t^{-1} \Pi_s^l = \Pi_s^l \si_{t-1}^{-1}$ for $l<t \le s$, so we have the equation $(\Pi_{q-1}^{q-r+1})^{-1} (\Pi_{q-1})^{l} = (\Pi_{q-1})^{l} (\Pi_{q-1-l}^{q-r+1-l})^{-1}$. In particular, if $l = q-r$, the right hand side of this equation is $(\Pi_{q-1})^{q-r} (\Pi_{r-1})^{-1}$, so we can rewrite the representative of $\beta$ as $\Pi_{q-1}^{r} \Pi_{q-2}^{r-1} \cdots \Pi_{q-r+1}^2  \Pi_{q-r}(\Pi_{q-1})^{q-r} (\Pi_{r-1})^{-nr+r} (\Pi_{q-1})^{p-q}$. The proof in the special case $n=1$ prescribes how $(\Pi_{r-1})^{-nr+r}$ and $(\Pi_{q-1})^{p-q}$ interact; when the two expressions are adjacent, $(\Pi_{r-1})^{-nr+r} (\Pi_{q-1})^{p-q}$, we can replace the word with $\Pi_{q-1}^{r} \Pi_{q-2}^{r-1} \cdots \Pi_{q-r+1}^2 \Pi_{q-r} (\Pi_{q-1}^{q-r+1})^{-nr+2r} (\Pi_{q-1})^{p-q-r}$. 

After repeating this process $l$ times, $\beta$ becomes $(\Pi_{r-1})^{-nr+r} (\Pi_{q-1})^{p-q} = \left(\Pi_{q-1}^{r} \Pi_{q-2}^{r-1} \cdots \Pi_{q-r+1}^2 \Pi_{q-r} (\Pi_{q-1})^{q-r}\right)^l (\Pi_{r-1})^{-nr+lr} (\Pi_{q-1})^{p-lq}$. When $p>nq$, we can set $l=n$, so $\beta = \left(\Pi_{q-1}^{r} \Pi_{q-2}^{r-1} \cdots \Pi_{q-r+1}^2 \Pi_{q-r} (\Pi_{q-1})^{q-r}\right)^n (\Pi_{q-1})^{p-nq}$. This braid is positive, so the closure of $\beta$, $K(p,q,r,-n)$, is a fibered knot.
\end{proof}

As the reviewer pointed out, one can see this theorem in a more geometric manner. A full twist on a braid with $q$ strands can be considered a little more simply in the following way. First the leftmost $r$ strands of the braid, where $1\le r \le p-1$, pass under the rightmost $q-r$ strands and then back over. The leftmost $r$ strands will still be on the left and the rightmost $q-r$ strands will still be on the right. We then twist the leftmost $r$ strands one full twist and twist the rightmost $q-r$ strands one full twist. Figure \ref{full twist} shows the full twist on 5 strands in the case $r=2$.

\begin{figure}[ht]
\begin{center}
\includegraphics[scale=.35]{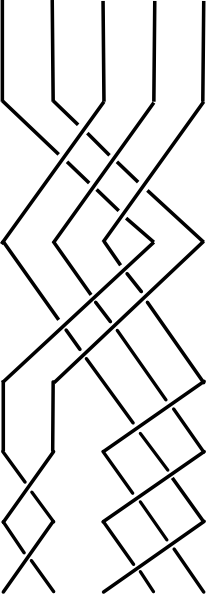}
\caption{A full twist on 5 strands}\label{full twist}
\end{center}
\end{figure}

Since we are considering $K(p,q,r, -n)$ with $n>0$, $r<q$ and $nq<p$, we are considering a braid in $B_q$ with at least $n$ full twists on all $q$ strands. If we express each full twist on all $q$ strands in the manner described above, the braid will have $r$ strands that pass over and under the rightmost $q-r$ strands but do not otherwise interact with those $q-r$ strands. Additionally, each of the $n$ full twists on all $q$ strands will contribute a full twist on those $r$ strands and after the $n$ full twists are completed, those $r$ strands will again be the leftmost $r$ strands. This means each of the $-n$ full twists on the leftmost $r$ strands prescribed by the definition of $K(p,q,r, -n)$ will easily cancel with one of the $+n$ full twists on the leftmost $r$ strands coming from the full twist on all $q$ strands. The example $K(12,5,2,-2)$ is show in Figure \ref{12-5-2-neg2}. Since the only negative crossings in the braid are in the $-n$ full twists on the leftmost $r$ strands, all of the negative crossings will disappear with this cancellation. Hence the the knot $K(p,q,r,-n)$ has a positive braid representative when $n>0$, $r<q$ and $nq<p$, so the knot is fibered.

As mentioned above, there is a simple example that shows that $K(p, q,r,-n)$ is not fibered for all values of the parameters $p$, $q$, $r$ and $n$. However, a more general statement about fibered twisted torus knots than the one given here is desirable.

\section{P/p-p/S and p/S-p/S knots}\label{bigonesection}

Primitive/primitive and primitive/Seifert knots, introduced by Berge \cite{berge} and Dean \cite{dean}, lie on the standard genus 2 Heegaard surface for $S^3$, so they have a natural associated slope, called the surface slope.

\begin{definition}
Let $N$ be a tubular neighborhood of $K$ in $S^3$. The surface slope of $K$ in $F$ is the isotopy class in $\partial N$ of one component of $\partial N \cap F$. 
\end{definition} 

Berge recognized that the primitive/primitive knots have lens space surgeries at the surface slope. Work of Dean \cite{dean} and Eudave-Mu\~noz \cite{eudavemunoz} shows that surgery on a primitive/Seifert knot at the surface slope is one of two types of Seifert fibered space, $S^2(a, b, c)$ or $\mathbb{RP}^2(d,e)$, or the connected sum of two lens spaces. Eudave-Mu\~noz showed that the reducible case can only arise from surgery on a nonhyperbolic knot if the knot is strongly invertible \cite{emreducible}. The exceptional surgeries the primitive/Seifert knots admit make them an especially interesting family to study. 

We provide here an explanation of what primitive/primitive and primitive/Seifert knots are. Consider a simple closed curve $K$ lying in the standard genus 2 Heegaard surface $F$ for $S^3$, and let $H$ and $H'$ be the handlebodies bounded by $F$. The process of 2-handle addition to $H$ along $K$, denoted $H[K]$, 

\begin{figure}[H]
\begin{center}
\includegraphics[scale=.30]{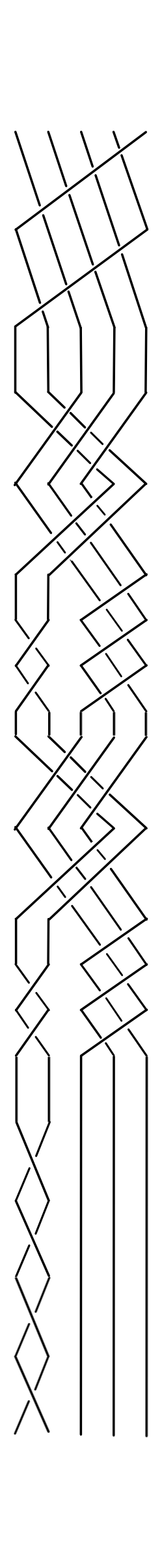}
\caption{$K(12,5,2,-2)$}\label{12-5-2-neg2}
\end{center}
\end{figure}

\noindent is gluing a $D^2\times I$ to $H$ so that $(\partial D^2) \times I$ is identified with a neighborhood of $K$ in $\partial H$.

\begin{definition}
$K$ is called \textit{primitive with respect to $H$} if $H[K]$ is homeomorphic to a solid torus.  
\end{definition}

\begin{definition}
$K$ is called \textit{Seifert with respect to $H$} if $H[K]$ is homeomorphic to a Seifert fibered space.
\end{definition}

We can also say $K$ is primitive with respect to $H'$ and Seifert with respect to $H'$ if $H'[K]$ is a solid torus and a Seifert fibered space, respectively. This allows us to define primitive/primitive and primitive/Seifert.

\begin{definition}
$K$ is called \textit{primitive/primitive with respect to $F$} if both $H[K]$ and $H'[K]$ are homeomorphic to a solid torus.
\end{definition}

\begin{definition}
$K$ is called \textit{primitive/Seifert with respect to $F$} if one of $H[K]$ and $H'[K]$ is homeomorphic to a solid torus and the other is homeomorphic to a Seifert fibered space.
\end{definition}

We often drop the phrase ``with respect to $F$" and refer to the knots as primitive/primitive and primitive/Seifert.

In unpublished work, Berge has shown that if a knot $K$ has two primitive/primitive representatives, $K_1$ and $K_2$, with the same surface slope, then there is a homeomorphism of $S^3$ sending the pair $(F, K_1)$ to the pair $(F, K_2)$. That is, if a given knot $K$ has a primitive/primitive representative $K_1$ with surface slope $s$, $K_1$ is the only representative of that knot with surface slope $s$. It contrast, the author has shown, in \cite{mypaper}, that there are knots which have distinct primitive/Seifert representatives with the same surface slope. Here we call such knots p/S-p/S. Additionally, \cite{mypaper} demonstrates that there are knots that have a primitive/Seifert representative and primitive/primitive representative so that both representatives have the same surface slope. We call these p/p-p/S. In fact, the two families of p/S-p/S and p/p-p/S knots found in \cite{mypaper} are part of the same phenomenon and can be considered special cases of the following theorem.

\begin{theorem}\label{thebigone}
The knots $K_1= K(kq+m, q, m, -1)$ and $K_2=K(kq +q-m, q, q-m, -1)$, where $q\ge 2$, $k \ge 2$, $1\le m \le q-1$ and $(q,m)=1$, are isotopic as knots in $S^3$ and have the same surface slope with respect to $F$, but there is
no homeomorphism of $S^3$ sending the pair $( F, K_1)$ to $( F, K_2)$
\end{theorem}

Theorem 3.1 of \cite{mypaper} is a special case of this theorem, where $q\ge5$ is odd and $m=\frac{q-1}{2}$, and Theorem 4.1 of \cite{mypaper} is the special case $m=1$. The proofs of those theorems employ distinct methods which are difficult to integrate. Here, we use the result of the braid calculation from Theorem \ref{fiberedttks} above. Dean's work \cite{dean} on the overlap between twisted torus knots and primitive/Seifert knots tells us that the knots described in this theorem are primitive/Seifert for $m>1$.

\begin{proof}[Proof of Theorem \ref{thebigone}]
By interchanging $m$ and $q-m$ if necessary, we may assume that $m<q/2$.
Because $K_1$ and $K_2$ are twisted torus knots, they are closures of braids in $B_q$,  represented by $\beta_1 = (\mbox{rev} \Pi_{q-1})^{kq+m} (\mbox{rev} \Pi_{m-1})^{-m}$ and $\beta_2 = (\mbox{rev}\Pi_{q-1})^{kq+q-m} (\mbox{rev} \Pi_{q-m-1})^{-q+m}$, respectively. We will show that $\beta_1$ and $\beta_2$ are conjugate in $B_q$ by $\gamma =\mbox{rev} \Delta_{m-1} \mbox{rev} \Delta_{q-1}^{m+1}$, i.e.  the braid $\beta_1 \gamma$, represented by $(\mbox{rev} \Pi_{q-1})^{kq+m} (\mbox{rev} \Pi_{m-1})^{-m} \mbox{rev} \Delta_{m-1} \mbox{rev} \Delta_{q-1}^{m+1}$, is the same in $B_q$ as $\gamma \beta_2$, represented by $\mbox{rev} \Delta_{m-1} \mbox{rev} \Delta_{q-1}^{m+1}(\mbox{rev}\Pi_{q-1})^{kq+q-m} (\mbox{rev} \Pi_{q-m-1})^{-q+m}$. Once again, we use the reverse words to simplify notation. After reversing the braid words, the goal is to show that  $\Delta_{q-1}^{m+1} \Delta_{m-1} (\Pi_{m-1})^{-m} (\Pi_{q-1})^{kq+m}$ represents the same element of $B_q$ as $(\Pi_{q-m-1})^{-q+m} (\Pi_{q-1})^{kq+q-m} \Delta_{q-1}^{m+1} \Delta_{m-1}$. (We will ignore that we have reversed the words and refer to rev$\beta_i$ and rev$\gamma$ as $\beta_i$ and $\gamma$, respectively, for the sake of easier notation. Because the words are reversed, our new braids are $\gamma \beta_1$ and $\beta_2 \gamma$.)

By the proof of the special case $n=1$ in Theorem \ref{fiberedttks}, $\beta_1 $ is represented by the braid word $\Pi_{q-1}^{m} \Pi_{q-2}^{m-1} \cdots \Pi_{q-m+1}^2  \Pi_{q-m} (\Pi_{q-1})^{kq}$ and $\beta_2$ is represented by $\Pi_{q-1}^{q-m} \Pi_{q-2}^{q-m-1} \cdots \Pi_{m+1}^2  \Pi_{m} (\Pi_{q-1})^{kq}$. Since $(\Pi_{q-1})^{q}$ is central in $B_q$ and appears in each of $\beta_1$ and $\beta_2$ with the same exponent, the braid words will be conjugate by $\gamma$ exactly when the word  $ \gamma \beta_1 = \Delta_{q-1}^{m+1} \Delta_{m-1} \Pi_{q-1}^{m} \Pi_{q-2}^{m-1} \cdots \Pi_{q-m+1}^2  \Pi_{q-m}$ is equal to the word $ \beta_2 \gamma=  \Pi_{q-1}^{q-m} \Pi_{q-2}^{q-m-1} \cdots \Pi_{m+1}^2  \Pi_{m} \Delta_{q-1}^{m+1} \Delta_{m-1}$. (Again, we ignore the change in the words and refer to $ \gamma \beta_1 (\Pi_{q-1})^{-kq}$  and $\beta_2 \gamma (\Pi_{q-1})^{-kq}$ as $\gamma\beta_1$ and $ \beta_2 \gamma$, respectively.)

We first consider $\gamma \beta_1$. Because $\Delta_{q-1}^{m+1}$ and $\Delta_{m-1}$ commute, $\gamma\beta_1$ can be rewritten as $ \Delta_{m-1} \Delta_{q-1}^{m+1} \Pi_{q-1}^{m} \Pi_{q-2}^{m-1} \cdots \Pi_{q-m+1}^2  \Pi_{q-m}$. By Lemma \ref{L1}, the generators that make up $\Delta_{q-1}^{m+1}$ will drop in index by one each time we pass them through a $\Pi_{i}^j$ in $\Pi_{q-1}^{m} \Pi_{q-2}^{m-1} \cdots \Pi_{q-m+1}^2  \Pi_{q-m}$. Then the index of each generator will drop by $m$ if we pass $\Delta_{q-1}^{m+1}$ all the way to the right of the word, and $\gamma \beta_1$ becomes $ \Delta_{m-1} \Pi_{q-1}^{m} \Pi_{q-2}^{m-1} \cdots \Pi_{q-m+1}^2  \Pi_{q-m} \Delta_{q-m-1}$. By the definition of $\Delta_{m-1}$, $\gamma\beta_1 =\Pi_{m-1} \Pi_{m-2} \cdots \Pi_1 \Pi_{q-1}^{m} \Pi_{q-2}^{m-1} \cdots \Pi_{q-m+1}^2  \Pi_{q-m} \Delta_{q-m-1}$. Using the commutativity properties of non-adjacent generators of $B_q$, we slide each $\Pi_i$ in $\Pi_{m-1} \Pi_{m-2} \cdots \Pi_1$ to the right until it is directly left of $\Pi_{q-m+i}^{i+1}$. Now $\gamma \beta_1$ is written as $\Pi_{m-1} \Pi_{q-1}^m \Pi_{m-2} \Pi_{q-2}^{m-1} \cdots \Pi_{1} \Pi_{q-m+1}^2 \Pi_{q-m} \Delta_{q-m-1}$. Because $\Pi_{l-1} \Pi_s^l = \Pi_s$, we write $\gamma \beta_1$ as $\Pi_{q-1}\Pi_{q-2} \cdots\Pi_{q-m+1}\Pi_{q-m}\Delta_{q-m-1} = \Delta_{q-1}$. Then the goal is to show $\beta_2 \gamma$ is also equal to $\Delta_{q-1}$.

Since $ \beta_2 \gamma = \Pi_{q-1}^{q-m} \Pi_{q-2}^{q-m-1} \cdots \Pi_{m+1}^2  \Pi_{m} \Delta_{q-1}^{m+1} \Delta_{m-1}$,  we examine how each generator appearing in $\Delta_{q-1}^{m+1}$ interacts with $\Pi_{q-1}^{q-m} \Pi_{q-2}^{q-m-1} \cdots \Pi_{m+1}^2  \Pi_{m}$. By the definition of $\Delta_s^l$, $\Delta_{q-1}^{m+1} = \Pi_{q-1}^{m+1} \Delta_{q-2}^{m+1}$. Thus, we begin with the word $\Pi_{q-1}^{q-m} \Pi_{q-2}^{q-m-1} \cdots \Pi_{m+1}^2  \Pi_{m} \si_{m+1}$. By Lemma \ref{lemma!}, $\Pi_s^{k+1} \Pi_s^k = \si_k \Pi_s^{k+1} \Pi_{s-1}^k$, so we can say that $\Pi_{m+1}^2 \Pi_m \si_{m+1} = \Pi_{m+1}^2 \Pi_{m+1} = \si_1 \Pi_{m+1}^2 \Pi_{m}$. Since $\si_1$ commutes with all of the generators appearing before it in $\beta_2\gamma$, we rewrite the word as $ \si_1 \Pi_{q-1}^{q-m} \Pi_{q-2}^{q-m-1} \cdots \Pi_{m+1}^2  \Pi_{m} \Pi_{q-1}^{m+2} \Delta_{q-1}^{m+1} \Delta_{m-1}$. 

This process is repeated with each generator in $\Pi_{q-1}^{m+2}$. That is, for $l \in \{m+2, \cdots, q-1\}$, $\si_l$ commutes with the $\Pi_i^j$ in $\Pi_{q-1}^{q-m} \Pi_{q-2}^{q-m-1} \cdots \Pi_{m+1}^2  \Pi_{m}$ for $i<l-1$. When we have pushed $\si_l$ to the left as far as possible, it is directly on the right of $\Pi_{l-1}^{l-m}$. Then $\beta_2 \gamma$ will contain a subword of the form $\Pi_{l}^{l-m+1} \Pi_{l}^{l-m}$, and Lemma \ref{lemma!} gives that this subword is equal to $\si_{l-m} \Pi_{l}^{l-m+1} \Pi_{l-1}^{l-m}$. Since $\si_{l-m}$ has index at least two smaller than every generator appearing in the word $\Pi_{q-1}^{q-m} \cdots \Pi_{l+1}^{l-m+2}$, $\si_{l-m}$ can be moved to the left of  $\Pi_{q-1}^{q-m} \cdots \Pi_{l+1}^{l-m+2}$ in $\beta_2 \gamma$. The final step of this process is moving $\si_{q-1}$ until it is on the right of $\Pi_{q-1}^{q-m} \Pi_{q-2}^{q-m-1}$. Since $\Pi_{q-1}^{q-m} \Pi_{q-2}^{q-m-1}\si_{q-1} = \Pi_{q-1}^{q-m} \Pi_{q-1}^{q-m-1}$, we can use Lemma \ref{lemma!} to say $\beta_2 \gamma= \Pi_{q-m-1} \Pi_{q-1}^{q-m} \Pi_{q-2}^{q-m-1} \cdots \Pi_{m+1}^2  \Pi_{m} \Delta_{q-2}^{m+1} \Delta_{m-1}$.

Now write $\Delta_{q-2}^{m+1}$ as $\Pi_{q-2}^{m+1} \Delta_{q-3}^{m+1}$. Using the process described in the previous paragraph, we see $\Pi_{q-2}^{q-m-1} \cdots \Pi_{m+1}^2  \Pi_{m} \Pi_{q-2}^{m+1} = \Pi_{q-m-2} \Pi_{q-2}^{q-m-1} \cdots \Pi_{m+1}^2  \Pi_{m}$. For each $m+1\le l\le q-1$, $\Pi_{l}^{m+1}$ can be pushed across $\Pi_{l}^{l-m} \cdots \Pi_m $ at the cost of changing the index of the generators so that for $m+1 \le i \le l$, the index $i$ is changed to $i-m$. In $\Pi_{q-1}^{q-m} \cdots \Pi_m$, there are $q-m$ groupings of the type $\Pi_i^j$, and in $\Delta_{q-1}^{m+1}$, there are $q-m-1$ groupings of the type $\Pi_i^j$, so there will be no groupings of $\Delta_{q-1}^{m+1}$ left on the right of the word when this process is complete. Then $\beta_2 \gamma = \Pi_{q-m-1} \Pi_{q-1}^{q-m} \Pi_{q-m-2} \Pi_{q-2}^{q-m-1} \cdots \Pi_{1} \Pi_{m+1}^2 \Pi_m \Delta_{m-1}$. The subword  $\Pi_{q-m-1} \Pi_{q-1}^{q-m} \Pi_{q-m-2} \Pi_{q-2}^{q-m-1} \cdots \Pi_{1} \Pi_{m+1}^2 \Pi_m$ can be rewritten as $\Pi_{q-1}\Pi_{q-2}\cdots \Pi_{m+1} \Pi_m$. Now $\beta_2 \gamma = \Pi_{q-1}\Pi_{q-2}\cdots \Pi_{m+1} \Pi_m \Delta_{m-1} = \Delta_{q-1}$, as desired.

Now we have shown that the knots are isotopic in $S^3$. Each knot has surface slope $kq^2+mq - m^2$. A homeomorphism of $f:S^3\rightarrow S^3$ sending $(F, K_1)$ to $(F, K_2)$ has to send $H$ to either $H$ or $H'$.  For $m=1$, $H[K_1]$ and $H'[K_1]$ are both solid tori, but $H[K_2] \cong D^2(k, q-1)$, where each of $k$ and $q-1$ is at least 2. Then the homeomorphism of from $(F, K_1)$ to $(F, K_2)$ would extend to a homeomorphism between a solid torus and a nontrivial Seifert fibered space, an impossibility.  For $m>1$, $K_1$ is $(k,m)$-Seifert with respect to $H$, $K_2$ is $(k, q-m)$-Seifert with respect to $H$, and both of $K_1$ and $K_2$ are primitive with respect to $H'$. Since $f(H) = H'$ would mean $f|_H$ extends to a homeomorphism of the solid torus to a nontrivial Seifert fibered space, $f$ must send $H$ to itself. Then the homeomorphism from $(F, K_1)$ to $(F, K_2)$ would extend to a homeomorphism of $D^2(k,m)$ and $D^2(k, q-m)$. Since $q-m>m$, there can be no such homeomorphism. Hence no homeomorphism $f:S^3 \rightarrow S^3$ can send $(F,K_1)$ to $(F,K_2)$.

\end{proof}

Since all of the p/p-p/S knots in Theorem \ref{thebigone} are torus knots,  one wonders if all p/p-p/S knots are torus knots, or whether some are hyperbolic. As mentioned above, recently Eudave-Mu\~noz, Miyazaki and Motegi showed, through different methods, that there is another family of p/S-p/S knots that are different than the ones shown here. Their findings suggest that these two families of p/S-p/S knots are not alone; there are likely more families of p/S-p/S knots to be found. 


\end{document}